\documentclass[11pt,reqno]{amsart}
\usepackage{graphicx,tikz}
\usepackage{amsfonts,amsmath,amssymb}
\usetikzlibrary{matrix,arrows}
\newcommand{\tmop}[1]{\ensuremath{\operatorname{#1}}}
\renewcommand{\Re}{\tmop{Re}}

\newcommand{\D}{\mathbb D}
\newcommand{\C}{\mathbb C}
\newcommand{\R}{\mathbb R}

\newtheorem{theorem}{Theorem}[section]
\newtheorem{lemma}[theorem]{Lemma}
\newtheorem{prop}[theorem]{Proposition}
\newtheorem{cor}[theorem]{Corollary}

\newtheorem{defn}[theorem]{Definition}

\newcommand{\Z}{\mathbb{Z}}%------------------------Integers

\begin{document}

\title {Applications of singular connections in symplectic and almost complex geometry}

\author{ Emmanuel Mazzilli*  and Alexandre Sukhov**}

\maketitle

\begin{abstract}  In this paper, we give two direct applications of the theory of singular connections developped by Harvey-Lawson \cite{HL1}. The first one is a version of Lelong-Poincar\' e formula for vector bundle over an almost complex manifold. The second is a convergence theorem for divisors associated to symplectic submanifolds constructed by Auroux in \cite{AU}. The case of hypersurfaces was done by Donaldsson in \cite{D}
\end{abstract}

\let\thefootnote\relax

*Universit\'e des Sciences et Technologies de Lille, 
U.F.R. de Math\'ematiques, 59655 Villeneuve d'Ascq, Cedex, France.
The author is partially supported by Labex CEMPI.

e-mail: Emmanuel.Mazzilli@math.univ-lille1.fr

**Universit\'e des Sciences et Technologies de Lille, 
U.F.R. de Math\'ematiques, 59655 Villeneuve d'Ascq, Cedex, France.
The author is partially supported by Labex CEMPI.

e-mail: sukhov@math.univ-lille1.fr

\section{Introduction}
The classical Lelong - Poincar\'e formula gives an equation between the  connection form of a holomorphic line bundle, the curvature and the current of integration over the zero locus of a holomorphic section. This formula plays a fundamental role in the complex geometry. The goal of the present paper is to extend this equation for currents to the case of complex vector bundles over symplectic almost complex manifolds. If an almost complex structure is not integrable, non-constant  holomorphic sections usually do not exist. For this reason we mainly work with almost holomorphic sections constructed by Donaldson in  his fundamental paper \cite{D} for the case of line bundles; his construction was extended to the case of arbitrary bundles by Auroux \cite{AU}.
Zero sets of such sections are symplectic hypersurfaces; this fact  explains the importance of Donaldson's construction for symplectic geometry. Another source of our inspiration is the important work of Harvey-Lawson \cite{HL1} where the Lelong-Poincar\'e formula is obtained in a very general situation of smooth sections of complex (not necessarily holomorphic) vector bundles.

In the present paper we adapt the approach of Harvey-Lawson to the special case of almost holomorphic sections. This leads to the equations for currents which generalize usual Lelong-Poincar\'e formula to the almost complex case. Note that the supplementary (with respect to the classical case) terms appearing in the equations can be controlled by suitable upper estimates. As an   application, in the last section we extend Donaldson's result on convergence of currents of integration over zero sets of almost holomorphic sections of line bundles to the case of sections of arbitrary bundles. We note that the first step toward the Lelong-Poincar\'e formula in the almost complex setting was done in the interesting paper by Elkhadra \cite{EL}. He considered the special case of trivial bundle with flat connection over an almost complex manifold and  the  currents of integration over (almost) complex submanifolds.

\tableofcontents

\section{Preliminaries}

For convenience of readers we recall some basic  facts about currents  and the geometry of complex vector bundles following \cite{DE,F, HL1, HS}. We also set up notations and terminology used through this paper.

\subsection{Currents}
Let $M$ be a smooth real manifold of dimension $m$. All manifolds considered in this subsection are oriented. Denote by $A^p(M)_s$ the space of  differential $p$-forms of class ${C}^s$ on $M$. We write $A^p(M)$ for $C^\infty$ forms. If $U$ is a chart with coordinates $x = (x_1,...,x_m)$, then  a $p$-form $\phi$ has a representation 
$\phi(x) = \sum_{\vert I \vert} \phi_I dx_I$ with the standard multi-index notation $I = (i_1,...,i_p)$, $i_1 < ...< i_p$ and $dx_I = dx_{i_1} \wedge ... \wedge dx_{i_p}$. Here $\vert I \vert = p$ (the length). For $\alpha = (\alpha_1,...,\alpha_m)$ consider the derivation operator  $D^\alpha = \partial^{\vert \alpha \vert}/\partial x_1^{\alpha_1}...\partial x_m^{\alpha_m}$ of order $\vert\alpha \vert = \alpha_1 + ...+ \alpha_m$. Given compact subset $K$ in $U$ and a positive integer $s$ define the standard seminorm

\begin{eqnarray}
\label{seminorm1}
p_K^s(\phi) = \sup_{x \in K} \max \{ \vert D^\alpha \phi_I(x) \vert  : \vert I \vert = p, \vert\alpha\vert \le s \} 
\end{eqnarray}
When $s$, $K$ and $U$ vary, the family of seminorms $p_K^s$ define the natural topology on $A^p(M)$. If $s$ is fixed, we obtain the topology on $A^p(M)_s$. If $K$ is a fixed compact subset of $M$, then ${\mathcal D}^p(K)$ denotes the subspace of $A^p(M)$ consisting of forms with support containing in $K$. The space of test forms on $M$ is defined by ${\mathcal D}^p(M) = \cup_K {\mathcal D}^p(K)$. Similarly can be defined the spaces of s-smooth p-forms ${\mathcal D}^p(K)_s$ and ${\mathcal D}^p(K)_s$. Then $A^p(M)$ (resp. $A^p(M)_s$) is a Fr\'echet space (resp. a Banach space). A current $T$ of dimension $p$ is a linear form on ${\mathcal D}^p(M)$ such that its restriction on every ${\mathcal D}^p(K)$ is continuous; the space of currents is denoted by ${\mathcal D}_p'(M)$. The same space is also denoted by ${{\mathcal D}'}^{m-p}(M)$ where $m-p$ is called  the degree of current. Similarly the topological dual of ${\mathcal D}^p(M)_s$ is called the space of currents of order (of singularity) $s$. This space can be identified with  the space of currents which are continuous on ${\mathcal D}^p(K)$ with respect to the seminorm $p^s_K$ for all compact subsets $K$.

Denote by $\langle T, \phi \rangle$ the action of a current $T$ on a test form $\phi$. The support $Supp T$ of $T$ is defined as in the usual distribution theory. Recall two basic examples playing the crucial role in the present work.

{\bf Example 1.} Let $Z$ be a closed oriented submanifold (may be with boundary) in $M$ of dimension $p$. The {\it current of integration} $[Z]$ over $Z$ 
is defined by 
$$\langle [Z],\phi \rangle = \int_Z \phi, \,\,\, \phi \in {\mathcal D}^p(M)$$
This is a current of order $0$ and of dimension $p$; its support coincides with $Z$.

{\bf Example 2.} A differential form $u$ of degree $q$ on $M$ with coefficients of class $L^1_{loc}(M)$ defines the current $T_u$ of dimension $m-q$:
$$\langle T_u, \phi \rangle = \int_M u \wedge \phi, \,\,\, \phi \in {\mathcal D}^{m-q}(M)$$
This current is of degree $q$ and of order $0$.  The correspondence $u \mapsto T_u$ is injective.

This last example is general. If $(x_1,...,x_m)$ are local coordiantes, every current  $T \in {{\mathcal D}'}^q(M)$ of order $s$ can be written in a unique way as a differential $q$-form
$$T = \sum_{\vert I \vert = q} T_I dx_I$$
whose coefficients are distributions of order $s$. The action of $T$ on $\phi \in {\mathcal D}^{m-q}(M)$ can be written in the form
$$ \langle T, \phi \rangle = \int_M T \wedge \phi$$
Here by convention the product of a coefficient $T_I$ of the current $T$ and a coefficient $\phi_J$ of the form $\phi$ is distribution's action $\langle T_I,\phi_J\rangle$.
The extrerior derivative  can be defined for a current $T$ of degree $q$  by duality putting 
$$\langle dT, \phi \rangle = (-1)^{q+1} \langle T, d \phi \rangle$$
For a current  $T \in {{\mathcal D}'}^q(M)$ and a smooth form $u \in A^r(M)$ the wedge-product is defined by 
$$\langle T \wedge u, \phi \rangle = \langle T, u \wedge \phi \rangle, \, \, \, \phi \in {{\mathcal D}'}^{m-q-r}(M)$$

Let $M_1$ and $M_2$ be oriented manifolds of dimensions $m_1$ and $m_2$ respectively. Consider a smooth map $f:M_1 \longrightarrow M_2$. 
Let $T \in {{\mathcal D}'}^{q}(M_1)$ be a current such that the restriction  $f \vert_{Supp T}$ is a proper map. Then the direct image $f_*T$ is defined by $\langle f_*T,\phi \rangle = \langle T, f^*\phi \rangle $ for $\phi \in {{\mathcal D}'}^{m_1 - q}(M_2)$. Note that $d(f_*T) = f_*(dT)$ and $f_*(T \wedge f^*u) = (f_*T) \wedge u$ for a smooth form $u$ on $M_2$. 

Suppose additionly that $f$ is a submersion. Let $u$ be a $q$-form  with coefficients of class $L^1_{loc}$ on $M_1$. Identifying as above the form $u$ and the current $T_u$ we can consider the direct image  $f_*u \in {{\mathcal D}'}^{q-(m_1 - m_2)}(M_2)$ which is of order singularity $0$. It follows by Fubini's theorem that 
\begin{eqnarray}
\label{Fubini}
f_*u(x) = \int_{t \in f^{-1}(x)} u(t)
\end{eqnarray}
This observation allows to define the inverse image $f^*T \in {{\mathcal D}'}^q(M_1)$ of a current $T \in {{\mathcal D}'}^q(M_2)$ as following:
$$\langle f^*T, \phi \rangle = \langle T, f_*\phi \rangle$$
Note that an important special case arises if $T = [Z]$. Then 
\begin{eqnarray}
\label{Zero_Back}
f^*[Z] = [f^{-1}(Z)]
\end{eqnarray}

\subsection{Zero divisors of atomic functions and sections} In this subsection we recall some facts from the theory developed in \cite{HL1,HS}. Let $M$ be a smooth oriented manifold of dimension $m$. Consider a current $T$ on $M$ viewed as a differential form with distributional coefficients.  We say that $T$ is of class $L^1_{loc}$ or is {\it  locally integrable} (on $M$) if every its coefficient is a function of class $L^1_{loc}$.
\begin{defn}
A smooth function $u:M \to \R^n$, $n > 1$, is called atomic if for each differential form $\alpha := dy^I/\vert y \vert^p$ on $\R^n$, with $p = \vert I \vert \le n-1$, the pullback $u^*(\alpha)$ admits an $L^1_{loc}$ extension acroos the zero set $Z_u:= u^{-1}(0)$. Furthermore, assume that $u$ does not vanish identically on any connected component of $M$. 
\end{defn}

If $u$ is atomic, the Lebesgue measure of $Z_u$ is equal to zero; in particular, the $L^1_{loc}$ extension is unique. We describe now geometric conditions which insure that a function is atomic.
They are local so we identify $M$ with an open subset of $\R^m$ with the standard coordinates $x$.
Suppose that $u: M \to \R^n$ is a function of class $C^{\infty}$ and satisfies the following conditions:
\begin{itemize}
\item[(i)] Algebraic vanishing. For each compact set $K \subset M$ there exist constants $c > 0$ (small) and $N$ (large) such that 
\begin{eqnarray}
\label{Loj}
\vert u(x) \vert \ge c \, dist(x,Z_u)^N, 
\end{eqnarray}
for all $x \in K$.
\item{(ii)} Strong codimension greater than $n-1$. The zero set $Z_u$ locally has the Minkowski codimension strictly greater than $n-1$ that is, for each compact set $K \subset M$, there exists an $\varepsilon > 0$ such that the upper Minkowski content of $Z \cap K$ in dimension $m-n+1 -\varepsilon$ is finite.
\end{itemize}
Then $u$ is  atomic.
In particular, if $u$ is real analytic and every irreducible component of $Z_u$ has codimension $\ge n$, then $u$ is atomic by Lojasiewicz inequalities.
Recall also that the upper Minkowski content $\bar M^k(X)$ of a (measurable) subset $X \subset \R^m$ in dimension $k$ is 
defined as following.  Given $\varepsilon > 0$, the $\varepsilon$-parallel set $X_\varepsilon$ is 
$X_\varepsilon = \{ x \in \R^m : \inf_{y \in X} \vert x - y \vert \le \varepsilon \}$. Then 
$\bar M^k(X) = \lim\sup_{\varepsilon \to 0} \varepsilon^{k-m}\lambda(X_\varepsilon)$ where $\lambda$ denotes the Lebesgue measure.

The divisor or zero current of an atomic function $u$ is defined as follows. Consider the solid angle potential on $\R^n$:
\begin{eqnarray}
\label{potential}
\theta = \sum_{j=1}^n (-1)^{j-1} \frac{y_jdy_1 \wedge ... \wedge \hat{dy_j} \wedge ... \wedge dy_n}{\vert y \vert^n}
\end{eqnarray}

If $u$ is an atomic function, the pullback $u^*(\theta)$ extends across $Z_u$ as an $L^1_{loc}$ current on $M$. Denote by $\sigma_n$ the volume of the unit sphere in $\R^n$. Then {\it the divisor} or {\it the zero current} $Div (u)$ of $u$ is defined by

\begin{eqnarray}
\label{divisor1}
Div(u) = d  \sigma_n^{-1} u^*(\theta) 
\end{eqnarray}
Thus, the divisor is a current of degree $n$. Note that $d  \sigma_n^{-1} \theta = [0]$, the current of integration over the origin $0$ in $\R^n$. Therefore one can consider $Div(u)$ as the pullback of $[0]$ by $u$, that is $u^*([0]) = Div(u)$.

Let $E$ be a smooth, oriented, real vector bundle of rank $n$ over a smooth oriented manifold $M$.
A section $s$ of $E$ is {\it atomic} if for each choice of local coordinates and frames, the corresponding $\R^n$-valued function is atomic. A point $x \in Z$ is called regular if $Z$ is a codimension $n$ Lipschitz submanifold near $x$; denote by $Z$ the zero set of $s$. Denote by  $Reg(Z)$ the set of regular points of $Z$ and by $Sing(Z) = Z \setminus Reg(Z)$ the set of singular points.
Let also $(Z_j)$ denotes the family of connected components of $Reg(Z)$. 

Suppose now that $s$ is an atomic section. Then for some integers $n_j$, 
\begin{eqnarray}
\label{divisor2}
Div(s) = \sum_j n_j [Z_j]
\end{eqnarray}
on $M \setminus Sing(Z)$.

Recall that a current of the form $\alpha + d\beta$, with $\alpha$, $\beta$ from $L^1_{loc}$ is called (locally) Federer {\it flat}. Clearly, $s^*(\theta)$ and $Div(\theta)$ are Federer flat currents.

Suppose additionly that the codimension-n Hausdorff  measure of $Sing(Z)$ vanishes and the 
current defined on $M \setminus Sing(Z)$ by the right hand of (\ref{divisor2}), has locally finite mass in $M$. Then (\ref{divisor2}) holds on $M$. In particular, this is the case if $s$ is a real analytic section and every irreducible component of $Z$ has codimension $\ge n$.

Recall that a current of the form $R + dS$ with $R$ and $S$ rectifiable currents, are called {\it integrally flat}.  Let $s$ be an  atomic section. Then the divisor $Div(s)$ is a $d$-closed, locally integrally flat current. Moreover, if $Div(s)$ has a locally finite mass, it is a locally rectifiable current.

 A current $T$ is called {\it normal} if $T$ and $dT$ are of order $0$.

\subsection{Geometry of vector bundles} Let $E$ be a smooth complex vector bundle of rank $n$ over a manifold $M$. We use the notation $A^p(M)$ for the space of (scalar) smooth differential $p$-forms on $M$ and $A^p(M,E)$ for the space of smooth $E$-valued p-forms.

 A {\it connection}  on $E$ is a $\C$-linear differential operator $D:A^0(M,E) \to A^1(M,E)$  satisfying the Leibnitz rule
\begin{eqnarray}
\label{connection1}
D (f s) = s df  + f  Ds
\end{eqnarray}
for every scalar forms $ \in A^0(M)$ and $s \in A^0(M,E)$.  Let $(e_1,...,e_n)$ be the frame corresponding to a local trivialization $U$ of $E$.
That is $e_i \in A^0(U,E)$ and $(e_1(x),...,e_n(x))$ is a basis of the fibre $E_x$ for every $x \in U$. We have
\begin{eqnarray}
\label{connection2}
D e_i = \sum  \omega_i^j e_j
\end{eqnarray}
where $\omega_i^j \in A^1(U,E)$ are scalar 1-forms. Viewing $e = (e_1,...,e_n)$ as a column we have in the matrix notation
\begin{eqnarray}
\label{connection3}
D e =  \omega e
\end{eqnarray}
The matrix $\omega = (\omega_i^j)$ with coefficients 1-forms is called  {\it the connection form} of $D$ (or {\it the gauge potential} of $D$).

 Then for every section $s \in A^0(U,E)$ we have the representation  $s = \sum s^j  e_j$ where $s_j$ are smooth scalar functions on $U$. Considering $s = {}^t(s_1,...,s_n)$ as a row vector, we obtain from (\ref{connection1}) and (\ref{connection2})  that $D$ acts on $s$ by
\begin{eqnarray}
\label{connection4}
Ds = d s +  s\omega
\end{eqnarray}
Let $\tilde{e} = (\tilde{e}_1,...,\tilde{e}_n)$ be another local frame over $U$. Then $e = \tilde{e} a$ where $a: U \longrightarrow GL(r;\C)$ is a matrix-valued function. Let $\tilde\omega$ be the connection form of $D$ with respect to $\tilde{e}$. Then $\tilde \omega$ follows the gauge transformation law:
$$\omega = a^{-1} \tilde \omega a + a^{-1} da.$$
In particular, the connection form is independent on the choice of trivialization and is defined globally on $E$.

 We extend the connection $D$ to a $\C$-linear map 
\begin{eqnarray}
\label{connection7}
D:A^p(M,E) \longrightarrow A^{p+1}(M,E)
\end{eqnarray}
setting
\begin{eqnarray}
\label{connection8}
D( f s) = (D s) \wedge f + s df
\end{eqnarray}
for any forms $s \in A^0(M,E)$ and $f \in A^p(M)$.

{\it The curvature tensor}  $\Omega = \Omega(D)$ is a global 2-form defined by the equation
\begin{eqnarray}
\label{connection5}
D^2s =\Omega s
\end{eqnarray}
In a local trivialization it can be expressed by
$$\Omega = d\omega + \omega \wedge \omega$$  
Note that in the case of a line bundle the second term in the right hand vanishes. If $\tilde{\Omega}$ is the curvature tensor in the trivialization $\tilde{e}$, then 
\begin{eqnarray}
\label{connection6}
\Omega = a^{-1} \tilde{\Omega} a
\end{eqnarray}

From now and on we assume that $(M,E)$ is a complex vector bundle over a complex manifold. Denote by $A^{p,q}(M)$ (resp. $A^{p,q}(M,E)$)  the spaces of scalar (resp. $E$-valued) smooth $(p,q)$-forms on $M$. Then the exterior derivative operator splits $d = d' + d''$ where $d':A^{p,q} \longrightarrow A^{p+1,q}$ and $d'': A^{p,q} \longrightarrow A^{p,q+1}$. This induces the splitting
$D = D' + D''$ where $D':A^{p,q}(M,E) \longrightarrow A^{p+1,q}(M,E)$ and $D'':A^{p,q}(M,E) \longrightarrow A^{p,q+1}(M,E)$. These operators satisfy 

$$D'( f s) = (D' s) \wedge f + s d'f$$
and 
$$D''( f s) = (D'' s) \wedge f + s d''f$$
for any forms $s \in A^0(M,E)$ and $f \in A^p(M)$. This is easy to check that the (0,2) component $\Omega^{0,2}$ of the curvature coincides with $D'' \circ D''$. We have the following fundamental criterion: $(M,E)$ admits a (unique) holomorphic vector bundle structure with the property $D'' = d''$ if and only if  $D'' \circ D'' = 0$. The unique connection on a holomorphic vector bundle with the property $D'' = d''$ is called the Chern connection.

A complex vector bundle $E$ is said to be hermitian if it is equipped with a  hermitian metric $h$ on every fibre smoothly depending on a fibre.
A connection $D$ is called { \it hermitian} if for any orthonormal frame  $i\omega$ is a hermitian 1-form that is $\omega^* = -\omega$. Here the star denotes the hermitian conjugation operator naturally extended to the space of matrices with 1-forms as entries. In particular, if 
 $E$ is a complex hermitian line bundle, then  a connection $D$ is hermitian if its  connection form associated to any orthonormal frame is a 1-form with purely imaginary values. In this case $i\omega$ and $i\Omega$ are real forms. The 2-form  $(i/2\pi)\Omega$ is real valued and closed hence represents a class in $H^2_{DR}(M,\R)$. This class is independent of $D$ and coincides with the image $c_1(E)_\R$ of first Chern class $c_1(E)$ under the natural morphism
$$H^2(M,\Z) \to H^2_{DR}(M,\R)$$
(the last cohomology group is canonically identified with $H^2(M,\R)$ by the De Rham
theorem).

Recall now the following classical example (see \cite{DE}) which is important for the present work. Assume that $M$ is oriented and $s$ is a smooth section of a line bundle $E$ which is of maximal rang 2 near the zero set $Z = s^{-1}(0)$. Then $Z$ is a closed submanifold of codimension 2 and with the orientation canonically induced by $M$ and $E$. Then the current $[Z]$ of integration over $Z$  belongs the same cohomology class as  the Chern class  $c_1(E)_{\R}$. Indeed, consider  the 1-form
\begin{eqnarray}
\label{Lelong1}
u = s^{-1}  Ds
\end{eqnarray}
defined on $M \setminus Z$. In local trivialization we have 
\begin{eqnarray}
\label{Lelong2}
 u = \frac{d s}{s} + \omega
 \end{eqnarray}
Note that $u$ has locally integrable coefficients near $Z$. The key argument is the following equality
\begin{eqnarray}
\label{Lelong3}
d\left ( \frac{d s}{s} \right ) = d \left ( s^*\frac{d\zeta}{\zeta} \right ) = s^*d \left ( \frac{d\zeta}{\zeta} \right ) = 
s^* (2\pi i \delta_0) = 2 \pi i [Z]
\end{eqnarray}
which holds in view of the Cauchy residue formula and (\ref{Zero_Back}). But $d\omega = \Omega$ and $(i/2\pi) \Omega = c_1(D)$. We have 
\begin{eqnarray}
\label{Lelong4}
(1/2\pi i)d u = [Z] - c_1(D)
\end{eqnarray}
Therefore $[Z]$ and $c_1(D)$ belong to the same class. Our goal is to extend the equation (\ref{Lelong4}) to wider classes of currents.

 \subsection{Hermitian connections over an almost complex manifold}
 
 We recall basic notions concerning almost complex manifolds.  Denote by $\D$  the
unit disc in $\C$ and by $J_{st}$  the standard complex structure
of $\C^n$; the value of $n$ will be clear from the context. We denote by $z = (z_1,...,z_n)$ the standard complex coordinates of $\C^n$.
A $C^1$ map  $f:(M',J') \to (M,J)$ between almost complex manifolds eqipped with almost complex structures $J'$ and $J$ respectively, is called $(J',J)$-holomorphic if satisfy the Cauchy-Riemann equations
\begin{eqnarray}
\label{CauchyRiemann0}
 df \circ J' = J \circ df
\end{eqnarray}
In the case whee $f$ is a diffeomorphism from $(M',J')$ to $M$ we can consider the diect image $f_{*}J' = df \circ J' \circ df^{-1}$.
It follows from (\ref{CauchyRiemann0}) that $f$ is $(J',f_{*}J')$ holomorphic. In the case where $f$ is a local coordinate map one can view $f_{*}J'$ as 
a coordinate representation of $J'$.

Two important special cases arise when $J'$ or $J$ is equal to $J_{st}$. We begin with the first case when $M' = \D$.
Then the equations (\ref{CauchyRiemann0}) take theform 
\begin{eqnarray}
\label{CauchyRiemann}
 df \circ J_{st} = J \circ df
\end{eqnarray}
A solution $f$  is called a $J$-holomorphic disc.
In local coordinates $z\in\C^n$, an almost complex structure
$J$ is represented by a $\R$-linear operator
$J(z):\C^n\to\C^n$, $z\in X^n$ such that $J(z)^2=-I$,
$I$ being the identity. Then the Cauchy-Riemann equations
(\ref{CauchyRiemann})  for a $J$-holomorphic disc $z:\D\to\C^n$, $z:\zeta \mapsto z(\zeta)$
can be written in the form
\begin{eqnarray}
\label{CauchyRiemann1}
z_\zeta + \mu(z) \overline{z_\zeta} = 0
\end{eqnarray}
where $\mu(z)$  a complex $n\times n$
matrix function uniquely detrmined by $J$.  We first recall the relation between $J$ and $\mu$ for
fixed $z$. Let $J:\C^n\to\C^n$ be a $\R$-linear map
so that $\det(J_{st}+J)\ne0$, where $J_{st} v=iv$.
Set $Q=-(J_{st}+J)^{-1}(J_{st}-J)$.
One can show that $J^2=-I$ if and only if
$QJ_{st}+J_{st} Q=0$, that is, $Q$ is a complex anti-linear operator.
Then there is a unique matrix $\mu \in Mat(n,\C)$ such that
$\mu v=Q\bar v$ for any $v\in\C^n$.
We introduce the manifolds of  linear maps and matrices $\mathcal J=\{J:\C^n\to\C^n: J\;{\rm is}\;\R{\rm-linear},\;J^2=-I,\;
\det(J_{st}+J)\ne0  \}$ and $\mathcal A=\{\mu \in Mat(n,\C): \det(I-\mu \bar \mu)\ne0 \}$.
One can show that  that the map $J\mapsto \mu$ is a birational
homeomorphism $\mathcal J\to\mathcal \mu$. Thus $J$ defines
a unique complex matrix function $\mu$ in $\Omega$ such that
$\mu(z)\in\mathcal A$, $z\in\Omega$. The matrix $\mu$ has the
same regularity properties as $J$.

Let now $z'$ be other local complex coordinates on $M$ and let $\mu'$ be the matrix function corresponding to  $J$ in these coordinates. 
It is easy to check that 
\begin{eqnarray}
\label{ComplexMatrix}
\mu'(z') = (z_z'\mu + z_{\overline z}')(\overline{z}_{\overline z}' + \overline{z}_z'\mu)^{-1} 
\end{eqnarray}
(see \cite{ST}). Note that $\mu(0) = 0$ if and only if $J(0) = J_{st}$. This condition always can be achieved by a linear change of coordinates. 
The following example will be used later. For $k > 0$ consider the dilations $\delta_k: z \mapsto z' = k^{-1/2}z$. Then 
$\mu'(z') = \mu(k^{-1/2}z')$. In particular $\vert \nabla \mu' \vert \le C k^{-1/2} \to 0$ as $k \to \infty$.

Consider now the second special case. Let $J$ be an almost complex structure
in a domain $\Omega\subset\C^n$.
Suppose $J(z)\in\mathcal J$, $z\in\Omega$. 
A $C^1$ function $f:\Omega\to\C$ is $(J,J_{st})$-holomorphic
if and only if it satisfies the Cauchy-Riemann equations
\begin{eqnarray}
\label{CRscalar}
f_{\bar z} - f_z \mu(z)  =0,
\end{eqnarray}
where $f_{\bar z} = (\partial f/\partial \overline{z}_1,...,\partial f/\partial \overline{z}_n)$ and $f_z = (\partial f/\partial {z}_1,...,\partial f/\partial {z}_n)$ are viewed as  row-vectors. Indeed, $f$ is $(J,J_{st})$ holomorphic if and only if 
for every $J$-holomorphic disc the composition $f \circ z$ is a usual holomorphic function  that is $\partial (f \circ z) /\partial\overline{\zeta} = 0$ on $\D$. Then the  Chain rule in combination with (\ref{CauchyRiemann1}) leads to (\ref{CRscalar}). Generally the only solutions to (\ref{CRscalar}) are constant functions  unless $J$ is integrable (then $\mu$ vanishes in suitable coordinates). Note also that (\ref{CRscalar}) is a linear PDE system while (\ref{CauchyRiemann}) is a quasilinear PDE for a vector function on $\D$.

Every $1$-form $\phi$ on $(M,J)$ admits a unique decomposition $\phi = \phi^{1,0} + \phi^{0,1}$ with respect to $J$. In particular, if $f:(M,J) \to \C$ is a $C^1$-complex function, we have $df = df^{1,0} + df^{0,1}$. We use the notation 
\begin{eqnarray}
\label{d-bar}
\partial_J f = df^{1,0} \,\,\,\mbox{and}\,\,\, \overline\partial_J f = df^{0,1}
\end{eqnarray}

In order to write this operator explicitely in local coordinates, we find a natural local basic in the space of (1,0) and (0,1) forms. We view  $dz = (dz_1,...,dz_n)^t$ and $d\overline{z} = (d\overline{z}_1,...,d\overline{z}_n)^t$ as vector-columns. Then the forms 
\begin{eqnarray}
\label{FormBasis}
\alpha = (\alpha_1,..., \alpha_n)^t = dz + \mu d\overline{z} \,\,\, \mbox{and} \,\, \overline\alpha = d\overline{z} + \overline\mu dz
\end{eqnarray}
form a basis in the space of  (1,0) and (0,1) forms respectively. Indeed, it suffices to note that for 1-form $\beta$ is (1,0) for if and only if for every $J$-holomorphic disc $z$ the pull-back $z^*\beta$ is a usual (1,0) form on $\D$. Using the equations (\ref{CauchyRiemann}) we obtain the claim.

Now we decompose the differential $df = f_zdz + f_{\overline{z}} d\overline{z} = \partial_J f + \overline\partial_J f $ in the basis $\alpha$, $\overline\alpha$ using (\ref{FormBasis}) and obtain the explicit expression 
\begin{eqnarray}
\label{d-bar2}
\overline\partial_J f = (f_{\overline{z}} (I - \overline{\mu}\mu)^{-1} - f_z (I - \mu\overline{\mu})^{-1})\overline\alpha
\end{eqnarray}

It is easy to check that the holomorphy condition $\overline\partial_J f = 0$ is equivalent to (\ref{CRscalar}). We note that the term $(I - \mu\overline\mu)^{-1}$ affects only constant in local estimates of the $\overline\partial_J$-operator which we will perfom in the next sections.

Consider the standard symplectic form
$$\Omega_0 = \frac{i}{2} \sum_{j=1}^n dz_j \wedge d\overline{z}_j$$
on $\C^n$. Then $\Omega_0 = id\omega$ with
$$\omega = \frac{1}{4}(\sum_{j=1}^n z_jd\overline{z}_j - \overline{z}_jdz_j)$$
Hence $-i\Omega_0 = d\omega$ is the curvature of a $U(1)$ connection on the trivial complex line bundle over $\C^n$ with connection matrix $\omega$. The connection $\omega$ defines a coupled $\overline{\partial}$-operator
$\overline{\partial}_\omega$ by
$$\overline{\partial}_\omega(f) = \overline{\partial} f + f\omega^{0,1}$$.

Now we bring together these two structures. Let $(M,\Omega)$ be a symplectic manifold with a 
{\it compatible} almost complex structure $J$. Consider a line bundle $L \to M$ with a  connection $\omega$ having the curvature $-i\Omega$. 
Then we can consider the $\overline\partial$-operator defined by $J$ and the connection matrix $\omega$. We set
\begin{eqnarray}
\label{d-bar5}
\partial_{J,\omega} f = Df^{1,0} \,\,\, \mbox{and} \,\,\, \overline\partial_{J,\omega} f = Df^{0,1}
\end{eqnarray}

Let $g$ be the Riemannian metric defined by $\Omega$ and $J$ that is $g(u,v) = \Omega(u,Jv)$.
This is well known that $g$ is an almost hermitian metric and $M$ is an almost K\"ahlerian manifold
(since $\Omega$ is closed).

 Let $E$ be a smooth complex vector bundle over an almost complex manifold $(M,J)$ of 
 complex dimension $n$. Denote by $A^{p,q}(M,E)$ the space of smooth sections of the bundle
 of $(p,q)$ forms valued in  $E$. A connection of type (1,0) on $E$ is a differential operator 
 $ \partial_{J,E}$ of order 1 acting on the complex $A^{ \bullet, \bullet}$ and satisfying the following two properies:
 
  \begin{itemize}
   \item[(i)]  $\partial_{J,E}: A^{p,q}(M,E)  \to A^{p+1,q}(M,E)$
    \item[(ii)] $ \partial_{J,E}(f  \wedge g) =  \partial_{J} f  \wedge g + (-1)^{deg f} f  \wedge  \partial_{J,E} g$
     \end{itemize}
     Similarly is defined the (0,1) connection $ \overline{ \partial}_{J,E}$. The operators d-bar introduced above are examples of these type of connections.

Let $(M,J)$ be an almost complex manifold of dimension $n$ and $N \subset M$ be a closed real $C^1$ submanifold of $M$ of real dimension $d$. The holomorphic tangent space at a point $x \in N$ is defined as 
$$H_xN = T_xN \cap J T_xN$$
Hence, $H_x$ is the maximal complex subspace of the tangent space $T_xN$. If the dimension of $H_xN$ is independent of $x$, then $N$  is called a CR submanifold and $\dim_\C H_xN$ is called the CR dimension of $N$. Recall also that $N$ is called generic (or generating) if the complex span of every $T_xN$ coincides with the whole $T_xM$.

The following proposition is well known in the case where an almost  complex structure is integrable ("the first structure theorem" for currents, see \cite{D}).

\begin{prop}
\label{FirstStrTh}
Let $N$ be a CRsubmanifold of CR dimension $k$ in $(M,J)$. Suppose that $T \in {\mathcal D}_{p,p}'(M)$ is a normal current such that $Supp T \subset N$ and $k < p$. Then $T = 0$.
\end{prop}
\begin{proof} This is easy to see that locally every CR manifold is contained in a generic manifold of the same CR dimension (which can be obtained by a removal of suitable local defining functions). Therefore without loss of generality we can assume that $N$ is generic. Fix a point of $N$. In local coordinates we can identify $M$ with a a neighborhood $U$ of the origin in $\C^n$; we can assume also that $J(0) = J_{st}$. After a $ \C$-linear change of coordinates one can also assume that $$N \cap U = \{ z \in U: \rho_j(z) = 0, j=1,...,n-k \}$$ and $\rho_j(z) = \Re z_{k+j} + o(\vert z \vert)$. Consider a local basis of (1,0) forms $\beta_j$ where $\beta_j = \alpha_j$ defined by (\ref{FormBasis}) when $j=1,...,k$ and $\beta_{k+j} = \partial_J\rho_{j-k}$, $j= k+1,...,n$. Then $$T = \sum_{\vert I \vert, \vert J \vert = n-p} T_{I,J} \beta_I \wedge \overline \beta_J.$$ Since the current $T$ is normal and $Supp T $ is contained in $N$, we conclude that $ \rho_j T = 0$ and $\rho_j \partial_J T  = 0$ for every $j$. Hence 
$\partial_J \rho_j \wedge T = \partial_J(\rho_j T) - \rho_j \partial_J T = 0$ for $j=1,...,n-k$ that is $\beta_j T = 0$ for $j=k+1,...,n$. Denote by $I'$ the multi-index which is the complement of the multi-index $I$. Since $p > k$, the multi-index  $I'$ contains at least one entry from  $\{ k+1,...,n\}$.
Hence $T \wedge \beta_{I'} \wedge \overline\beta_{J'} =0$ and $T_{I,J} = 0$.
\end{proof}

 \subsection{The fundamental equation}  Consider two complex line bundles $F$ and $L$ over a smooth manifold $M$. We suppose that $F = M \times \C$ is the trivial bundle. Then a smooth section $s: M \to L$ can be viewed as a bundle map  $\alpha: F \to L$ defined by 
 $\alpha(1) = s$ on each fiber.

 Let $D_L$ be a connection on $L$ with the gauge potential (the connection matrix)  $\omega_L$. Then we can define a new  connection $D$ on $F$ by $D = \alpha^{-1} D_L \alpha$. The gauge potential for $D$ is given by
 \begin{eqnarray}
 \label{gauge_pot}
 \tau = \frac{d\sigma}{\sigma} + \omega_L
 \end{eqnarray}
 where $\sigma$ is a representation of $s$ in a local trivialization.
 This is a well defined global 1-form outside the zero set of $s$. We suppose that $s$ is atomic. Then $\tau$ extends through $Z_s$ as a well defined global current which is Federer flat, i.e. admits a representation $\beta + d\gamma$ with $\beta,\gamma$  are in $L^1_{loc}$.
 
 The following result is due to Harvey-Lawson \cite{HL1}:
 
 \begin{prop}
 \label{HLtheo1}
 Let $E$ be a complex line bundle over $M$ with a connection $D_L$. Let $s$ be a smooth section of $L$ satisfying the following properties:
 \begin{itemize}
 \item[(i)]  The section $s$ vanishes to finite order.
 \item[(ii)] The zero set $Z$ of $s$ has Minkowski codimension greater than one.
 \item[(iii)] The set $Sing Z$ has Hausdorff measure zero in codimension-two.
 \item[(iv)] The current $Div(s)$ has locally finite mass.
 \end{itemize}
 Then $Div(s)$ is of the form 
 $$Div(s) = \sum_j n_j [Z_j]$$
 where $n_j$ are integers and the $Z_j$'s are oriented codimension-two components of $Reg Z$. Furthermore, we have the equation
 \begin{eqnarray}
 \label{LP1}
 \frac{1}{2\pi i} d\tau = \sum n_j [Z_j] - c_1(D_L)
 \end{eqnarray}
 \end{prop}
 The formula (\ref{LP1}) is a direct generalization of (\ref{Lelong4}). This fundamental equation gives a relation between the connection form, the curvature and the zero set of a section. In the special case of regular sections (i.e. sections which are submersions) this equation was used by Donaldson.

 \section{The Lelong-Poincar\'e formula for complex line bundles}
 
 This section is devoted to the  important special case of line bundles.

\subsection{Lelong-Poincar\'e formula for atomic sections}

Let  $(M,J)$ be an almost complex variety and let  $L\rightarrow M$ be a complex line bundle  over 
$M$. Fix a hermitian metric  $\langle\ ,\ \rangle$ on $L$ as well as a hermitian connection $D_L$ 
with the curvature  $\Omega$.

Let  $s$ be a smooth section of  $L$. Then the associated singular  connection defined in the previous subsection  
has the form  $Ds=\tau s$. In a local trivialization we have the representation (\ref{gauge_pot}). The connection form $\omega$ of $D_L$ can be decomposed into $\omega = \omega^{1,0} + \omega^{0,1}$. We assume that the local frame is orthonormal and therefore $i\omega$ is real that is 

\begin{eqnarray}
\label{LP2}
-\omega^{0,1} = \overline{\omega^{1,0}}
\end{eqnarray}
The current $\tau$ in its turn can be decomposed into the sum of  $(1,0)$  and $(0,1)$ forms denoted by  $\tau^{(1,0)}$ and $\tau^{(0,1)}$ respectively. Hence $\tau$ admits the decomposition 
$$\frac{d\sigma}{\sigma} + \omega  = \tau^{1,0} + \tau^{0,1}$$ 
with 
\begin{eqnarray}
\label{LP3}
\tau^{1,0} = \frac{\partial_J \sigma}{\sigma} + \omega^{1,0} \,\,\, \mbox{and} \,\,\,\tau^{0,1} = \frac{\overline{\partial}_J \sigma}{\sigma} + \omega^{0,1}
\end{eqnarray}

Let $s$ be a smooth atomic section of $L$. Since the section $s$ is atomic, the forms  $\tau$, $\tau^{(0,1)}$ et $\tau^{(1,0)}$ are locally integrable and define currents.

\begin{theorem}
\label{Prop_LP5}  We have the following equation for currents:
\begin{eqnarray}
\label{AlHolPL}
{1\over 2\pi i}d\ \partial_{J} \log \langle s,s\rangle =Div(s)+{1\over 2\pi i}d\overline{\tau^{(0,1)}}- {1\over 2\pi i}d\tau^{(0,1)}-{i\over 2\pi}\Omega. 
\end{eqnarray} 
\end{theorem}
Note  that $\log \langle s, s \rangle$ is a smooth function defined on $M \setminus Z$ where $Z = s^{-1}(0)$.

\begin{proof} Chosse a local orthonormal frame $e$ such that $s = \sigma e$. Then $\langle s, s \rangle = \sigma \overline{\sigma}$; We have

\begin{eqnarray*}
& &\partial_{J} \log \sigma \overline{\sigma} =  \frac{1}{\vert \sigma \vert^2} ( (\partial_{J} \sigma)\overline{\sigma} +  \sigma\overline{(\overline{\partial}_{J}\sigma)}) =\tau^{(1,0)}+\overline{\tau^{(0,1)}}=\tau +\overline{\tau^{(0,1)}}-\tau^{(0,1)}.
 \end{eqnarray*}
The second equality follows from the representations (\ref{LP3})  and the equality (\ref{LP2}). Multiplying by ${1\over 2\pi i}$ and taking the exterior derivative we obtain 
$${1\over 2\pi i}d\ \partial_{J}\log \vert \sigma \vert^2 ={1\over 2\pi i}d\tau + {1\over 2\pi i}d\overline{\tau^{(0,1)}}- {1\over 2\pi i}d\tau^{(0,1)}.$$ 

Since $s$ is an atomic section, we can apply now (\ref{LP1}). This concludes the proof.\end{proof}

Note that in general the operator  $d$ can not be replaced by  $\partial_J$ since for an almost complex (not integrable) structure $J$ we have  $\partial^{2}_J\not = 0$.

In the special case where an almost complex structure  $J$ is integrable and  $L$ is a hermitian holomorphic line bundle over $M$, we  apply Proposition \ref{Prop_LP5} with the Chern connection as $D_L$. Then we obtain the usual Lelong-Poincar\'e formula  for holomorphic sections of  $L$ :

\begin{eqnarray}
\label{PL6}
{1\over 2\pi i}\bar\partial _J\ \partial_J\log \langle s,s \rangle = Div(s)- c_1(L)
\end{eqnarray}
Indeeed, in this case $\partial^{2}_J=0$ and $\tau^{(0,1)}=0$ if $s$ is holomorphic. 

As another special case we refind the result of  Elkhadra \cite{EL} for sections $s$  holomorphic along the zero set  $Z = s^{-1}(0)$ that is $\bar\partial_{J,L}(s)=0$ on $Z$. Note that $Z$ supposed to be a smooth manifold of real codimensin $2$. Writing  $s=\sigma e$ where $e$ is a local frame of $L$, we have
$$\bar\partial_{J,L}(s)={\bar\partial_J (\sigma)}e + \sigma\bar\partial_{J,L}(e).$$ 
Therefore $\overline{\partial}_J \sigma=\sigma h_1+\overline{\sigma} h_2$ where $h_1$ and $h_2$ are smooth functions. Therefore  the currents $\overline{\partial}_{J,L} s / s$ and $d\tau^{(0,1)}$ are in  $L^1_{loc}$. Note that if $J$ is integrable and  $s$ is holomorphic on  $M$  then  $\tau^{(0,1)}=0$ and we obtain the  Lelong formula with a connection different from the Chern connection.

\subsection{Almost holomorphic sections}
We recall the definition of an almost holomorphic section in the sense of Donaldson. Let $(M,J)$ be an almost complex manifold and $L \longrightarrow M$ be a complex line bundle over $M$. Since $J$ is fixed, we simply  write $\partial_L$ instead of $\partial_{J,L}$.

\begin{defn} A smooth  section $s$ of $L$ is called ($\lambda$-) almost holomorphic if there exists   $0 \le \lambda<1$ such that the estimate 
\begin{eqnarray}
\label{D1}
\vert\bar\partial_{L}s\vert \leq \lambda\vert \partial_{L} s\vert 
\end{eqnarray}
holds at every point of  the zero set $Z = s^{-1}(0)$ \footnote{Usually  $\lambda$ will be small }.\end{defn}

We say that an almost holomorphic section $s$ is regular if it is of maximal rang 2 in a neighborhood of its zero set $Z$. In this case $Z$ is a 
real submanifold of $M$ of codimension 2. By  Donaldson,  $Z$ is a symplectic submanifold of $M$ if   $M$ is equipped with a symplectic structure  $\Omega$ such that $J$ is compatible with $\Omega$. 

Given $\delta > 0$ and a subset $X$ of $M$ we use the notations 
$$X^\delta = \{ p \in M: d(p,X) \le \delta \}$$
$$X_\delta = \{ p \in M: d(p,X) \ge \delta \}$$
We also denote by $d\nu$ a volume form on $M$.

\begin{prop} There exist constants $C_1 > 0$ and $C_2 > 0$ with the following property. For every  regular almost holomorphic section $s$ and    every  compact subset $K$ of $M$ there is (a sufficiently small ) $\varepsilon>0$ such that the current $d\tau^{0,1}$ from (\ref{AlHolPL}) admits the decomposition 
$$d\tau^{0,1} = T_{1}^{\varepsilon}+T_{2}^{\varepsilon}.$$ The currents  $T_{1}^{\varepsilon}$ and  $T_{2}^{\varepsilon}$ satisfy the following estimates: 
\begin{eqnarray}
\label{D2}
\vert\langle T_{1}^{\varepsilon},\phi\rangle\vert  \leq C_1 \lambda\vert\vert d\phi\vert\vert_{\infty}I_1
\end{eqnarray}
\begin{eqnarray}
\label{D3}
\vert \langle T_{2}^{\varepsilon},\phi\rangle\vert  \leq C_2\vert\vert \phi\vert\vert_{\infty}(I_2 + I_3)  
\end{eqnarray}
Here 
$$I_1 = \int_{Supp(\phi)\cap Z^\varepsilon}\left \vert{\partial_L s\over s }\right\vert d\nu$$
$$I_2 = \int_{Supp(\phi)\cap Z_{(1/2)\varepsilon}}\left\vert d\left({\bar\partial_L(s)\over s }\right)\right\vert d\nu$$
$$I_3 = (\lambda	/\varepsilon)\int_{Supp(\phi)\cap Z_{(1/2)\varepsilon}\cap Z^\varepsilon}\left \vert{\partial_L s\over s }\right\vert d\nu$$
with an arbitrary test form $\phi$ whose  support is contained in  $K$.
\end{prop}

%\footnote{j'utilise la notation de Donaldson pour $\tau:={Ds\over s}$, idem $\tau^{(0,1)}$ et $\tau^{(1,0)}$ : on voit mieux que le terme devient petit pr\`es de $X$...je trouve}.

We see that $T_{1}^{\varepsilon}$ is a current of order $1$ such that its norm tends to zero when $\lambda$ tends to zero i.e. when a section  $s$  is close to a holomorphic one along $Z$.  The current $T_{2}^{\varepsilon}$ is of order zero but its norm in general is not small when $\lambda$ is small.

\begin{proof} Since  $s$ is a regular section,  there exists $\varepsilon>0$ such that for every point  $p\in K\cap Z^\varepsilon$ one has  $\vert \bar\partial_L(s)\vert\leq 2\lambda\vert\partial_L(s)\vert$. Fix a smooth cut-off function  $0\leq \chi\leq 1$ with  support containing in  $K\cap Z^\varepsilon$  and equal to   1 on $K\cap Z^{(1/2)\varepsilon}$. We have 
$$\bar\partial_L(s)=\chi \bar\partial_L(s)+(1-\chi)\bar\partial_L(s).$$ From the equality  $\tau^{(0,1)}=\bar\partial_L(s)/s$ we obtain
$$\langle d\tau^{(0,1)},\phi\rangle = \langle \chi \bar\partial_L(s)/s,d\phi \rangle + \langle d((1-\chi )\bar\partial_L(s)/s),\phi \rangle = T_1^\varepsilon + T_2^\varepsilon.$$
Estimates (\ref{D2}) and (\ref{D3}) now follows from (\ref{D1}).\end{proof}

\section{The Lelong-Poincar\'e formula for vector bundles  }
In this section we consider the general case of vector bundles of arbitrary rank.

\subsection{Bochner-Martinelli kernel} We begin with the case of the trivial vector bundle $E \longrightarrow M$ of rang $n$ with the flat connection over an almost hermitian manifold $(M,J)$. Then every section of $E$ can be identified with a complex vector function $u = (u_1,...,u_n)$ (viewed as a vector row) on $M$. Denote by $h$ the matrix (with respect to the canonical basis of $\C^n$) of the hermitian metric on $M$.Set $u^* = h \bar u^t$ so that $\vert u \vert^2 = u u^*$. Let  $\beta$ be a form defined by  
\begin{eqnarray}
\label{BM1}
 \beta=\frac{du  u^*}{\vert u\vert^2}.
 \end{eqnarray}
 The form 
 \begin{eqnarray}
 \label{BM2}
B(u)=\left(1/2\pi i\right)^n \beta\wedge (d\beta)^{n-1}.
\end{eqnarray}
is called the Bochner-Martinelli form based on the metric $h$. In the case where $h$ is the identity matrix and $M$ is a complex manifold for any holomorphic sections, the form (\ref{BM2}) coincides with the standard Bochner-Martinelli kernel which can be written in terms of the standard complex structure as follows:
\begin{eqnarray}
\label{BM1.1}
B_{st}(u) = \left(i/2 \pi\right)^n \partial \log \vert u \vert^2 \wedge (\bar\partial\partial \log \vert u \vert^2)^{n-1}
\end{eqnarray}

 We use the following local result (\cite{HL1}, pp 173 Def.A.21, pp174 Prop. A.22 and Lemma A.23):
\begin{prop} 
In the above trivialization chart the following equation holds:
\begin{eqnarray}
\label{BM3}
 dB(u) = Div(s)
 \end{eqnarray}
\end{prop}

Let  $E \longrightarrow M$ be a  hermitian vector bundle of rang $n$ over  a manifold $M$. We suppose that $E \longrightarrow M$ is equipped with a  connection $D$ and denote by $\omega$ the connection form; we denote by $\Omega$ the curvature matrix.

Clearly the Bochner-Martinelli potential  (\ref{BM2})  is not defined intrinsecly and depends on a choice of a local frame $f$. Moreover it works only for a flat connection. By contrary its exterior derivative $d B$ is intrensic in view of (\ref{BM3}).

An important result of  \cite{HL1} presents a construction of a global potential  connecting with the divisor $s$ by means the connection $D$ on $E$ i.e; satisfying the analog of the fundamental equation for the case of higher rank bundle. Note that in \cite{HL1} the connection $D$ is not supposed to be  hermitian.  

Let $(f)$ be a local frame in a trivialization of $E$. For a section $s$ we use the notation as above:  $s=uf$. Consider $f^{*}$ the co-frame of $f$ and the exterior algebra of $E^{*} \oplus E$ oriented by the volume form $f_1^{*}\wedge f_1\cdots\wedge f_n^{*}\wedge f_n:=\lambda (f)$. If  $h$ is the matrix of the metric in this trivialization and if we note : $u^{*}=h\bar u^t$, $\vert u \vert^2 = u u^*$, $Du=du+u\omega$ and $Du^{*}=du^{*}-\omega u^{*}$, where  $\omega$ is the connection matrix  in this local frame, the desired global potential form is given by the following expression:

 %\footnote{le signe $-$ s'explique tr\`es facilement : si $D$ est hermitienne et $f$ est orthonorm\'ee alors dans cette base la matrice de la connexion est anti-hermitienne}. 

 \begin{eqnarray}
\label{BM4}
\sigma\lambda(f) ={1\over n!}\left(i/ 2\pi\right)^nf^ {*}{u^{*}Du\over \vert u\vert^2}f\sum_{k=0}^{n-1}\left(f^ {*}\Omega f -f^ {*}{Du^{*} \wedge Du\over \vert u\vert^2}f\right)^k\wedge \left(f^*\Omega f\right)^{n-1-k},\end{eqnarray}
where  $\Omega$ is the curvature form of $D$.

As usual, denote by  $c_n(D)$  the n-th Chern class of $D$. Recall the following result  (\cite{HL1} Theorem 4.22) 

\begin{prop} The form  $\sigma$ defined in a local trivialization by (\ref{BM4}) is defined globally on $M$ and satisfies the fundamental equation 
\begin{eqnarray}
\label{BM5}
 d\sigma=c_n(D)-Div(s)
 \end{eqnarray}
\end{prop}
%%Soit $Z$ les z\'eros de la sections $s$ dans $X$ et $1_{X/Z}d\sigma$ le courant d\'efini par $\lim \chi({\vert s\vert^2_{h}\over \varepsilon^2})d\sigma$ o\`u $\chi$ est une approximation lisse de l'indicatrice de $[1,\infty)$. D'apr\`es le th\'eor\`eme pr\'ec\'edent, nous avons l'\'egalit\'e entre courants:
%%$$d\sigma-1_{X/Z}d\sigma=-Div(s).$$
In the next subsection, we consider only the special case  where $M$ is equipped with an almost complex structure $J$ and a hermitian connection $D$ on $E$. We study the properties of this kernel.

\subsection{Estimates of the kernel}
 Our goal is to represent $d\sigma$ in a more explicit way by analogy with (\ref{BM1.1}). We assume additionly that a section $s$ is a submersion that is has the maximal rang $n$ in a neighborhood of its zero set $Z$. 

We denote again by $\langle, \rangle$ the canonical extension of the hermitian product to the exterior algebra. If $(e)$ is a local frame and $u \in A^p(M,E)$, $v \in A^q(M,E)$ are differential forms with $u = \sum_j u_j \otimes  e_j$, $v = \sum_k v_k \otimes e_k$, then 
$$\langle u, v \rangle = \sum_{j,k} u_j \wedge \bar v_k \langle e_j, e_k \rangle$$
Note that 
$$d \langle u, v \rangle = \langle D u, v \rangle + (-1)^p \langle u, D v \rangle$$
since the connection $D$ is hermitian.

We have the identities:
\begin{eqnarray*}
& &\partial_J \log \langle s, s \rangle = \frac{1}{\vert s \vert^2}\left ( \langle  \partial_{J,\omega} s, s \rangle + \langle s, \bar \partial_{J,\omega}s \rangle \right )\\ & & 
= \frac{1}{\vert s \vert^2}\langle Ds, s \rangle + \frac{1}{\vert s \vert^2} \left ( \langle s, \bar\partial_{J,\omega} s \rangle - \langle \bar\partial_{J,\omega} s, s \rangle \right)
\end{eqnarray*}

Now a direct computation gives
\begin{eqnarray}
\label{BM5}
d\partial_{J}\log \langle s, s\rangle = \tau_1 + d\tau_2
\end{eqnarray}
with 
\begin{eqnarray}
\label{BM6}
\tau_1 = \frac{1}{\vert s \vert^2}\left(\langle D^2 s, s \rangle - \langle Ds, Ds \rangle \right ) - 
\frac{\langle Ds,s  \rangle}{\vert s \vert^4}\left(\langle Ds,s \rangle + \langle s, Ds \rangle \right )
\end{eqnarray}
and 
\begin{eqnarray}
\label{BM7}
\tau_2 = \frac{1}{\vert s \vert^2} \left ( \langle s, \bar\partial_{J,\omega} s \rangle - \langle \bar\partial_{J,\omega} s, s \rangle \right)
\end{eqnarray}

%\begin{rem}c'est ici que j'ai besoin que $s$ n'ait que des z\'eros simples sur $Z$ car si $s$ est juste atomique, je ne sais pas montr\'e que $d\tau_{2}$ est $L^1_{loc}$ quand $d$ porte sur $\bar\partial_{A,J}(s)$ et je crois que cela ne l'est pas en g\'en\'eral...\end{rem}

We have:
\begin{eqnarray*}
& &\partial_{J}\log \vert s\vert ^2\wedge (d\partial_{J}(\log \vert s\vert^2))^{n-1}  = \partial_{J}\log\vert s\vert^2\wedge\sum_{k=0}^{n-1}C^{k}_{n-1}\tau_{1}^k\wedge (d\tau_{2})^{n-1-k} =\\
& & {\langle Ds,s\rangle \over \vert s\vert^2}\wedge \tau_{1}^{n-1}+\tau_{2}\wedge\tau_{1}^{n-1}+ \left({\langle Ds,s\rangle\over \vert s\vert^2}+\tau_ {2}\right)\wedge\sum_{k=0}^{n-2}C^{k}_{n-1}\tau_{1}^k\wedge (d\tau_{2})^{n-1-k}
\end{eqnarray*}
We study the  currents appearing in  this expansion and proceed in several steps.
\medskip

(1)  Note that $\langle Ds, s \rangle \wedge \langle Ds, s \rangle = 0$.  Hence it follows from (\ref{BM6}) that  
\begin{eqnarray*}
{\langle Ds,s\rangle\over \vert s\vert^2}\wedge \tau_{1}^{n-1} = {\langle Ds,s\rangle\over \vert s\vert^{2n}}\sum_{k=0}^{n-1}(-1)^{n-1-k}C^{k}_{n-1}\langle D^2s,s\rangle^k\wedge\langle Ds,Ds\rangle^{n-1-k}
\end{eqnarray*}
Fix an orthonormal basis $(f)$ in a local trivializatin such that $s=uf$. Recall that we have the identities  $$\langle Ds,s\rangle =du\bar u^{t}+u\omega\bar u^{t}, \,\,\ \langle D^2s,s\rangle =u\Omega\bar u^t$$ and
 $$\langle Ds,Ds\rangle =du \wedge d\bar u^t-du \wedge \omega \bar u^t+ u\omega \wedge d\bar u^t-u\omega \wedge \omega\bar u^t$$ 
 This leads to the following representation: 
$$d\left({\langle Ds,s\rangle\over \vert s\vert^2}\wedge \tau_{1}^{n-1}\right)=d(\beta\wedge (d\beta)^{n-1})+ T$$
where  $\beta$ is defined by (\ref{BM1}) with $h = Id$ and  $T$ is a current in  $L^1_{loc}$. Therefore 
\begin{eqnarray}
\label{BM8}
d\left({[Ds,s]\over \vert s\vert^2}\wedge \tau_{1}^{n-1}\right)=(2\pi i)^nDiv(s)+T
\end{eqnarray}
in the above trivialization. 
\medskip

(2) For every compact subset $K$ in  $M$ and every test form  $\phi$ with support in  $K$ there exists a constant  $C_K > 0$ such that  : 
\begin{eqnarray*}
\vert \left\langle d(\tau_2\wedge\tau_ 1^{n-1}),\phi\right\rangle\vert \leq C_K\left(\vert\vert \bar\partial_{J,\omega}s\vert\vert_{\infty}\vert\vert d\phi\vert\vert_{\infty}\right).
\end{eqnarray*}
A similar estimate holds for the terms d$\big(\tau_ {2}\wedge\sum_{k=0}^{n-2}C^{k}_{n-1}\tau_{1}^k\wedge (d\tau_{2})^{n-1-k}\big).$
\medskip

(3) Finally, we have  
\begin{eqnarray*}
& &\vert\langle d \left(\frac{\langle Ds,s\rangle}{ \vert s\vert^2}\wedge\sum_{k=0}^{n-2}C^{k}_{n-1}\tau_{1}^k\wedge (d\tau_{2})^{n-1-k}\right),\phi\rangle\vert\\
& &\leq C_K\left(\vert\vert \bar\partial_{J,\omega}s\vert\vert_{\infty}\vert\vert d\phi\vert\vert_{\infty}+\vert\vert \phi\vert\vert_{\infty}\right).
\end{eqnarray*}
%En r\'esum\'e : la diff\'erentielle ext\'erieure de la somme des potentiels $\sigma$ (le potentiel de \cite{HL1}, th\'eo. 4.22) et  $({1\over 2\pi i})^nd\big({[Ds,s]\over \vert s\vert^2}\wedge \tau_{1}^{n-1}\big)$ est dans $L^1_{loc}(X)$ car cette diff\'erentielle est la somme de $c_n(D)$ et de $1_{X/Z}({1\over 2\pi i})^nd\big({[Ds,s]\over \vert s\vert^2}\wedge \tau_{1}^{n-1}\big)$. Nous avons donc montr\'e la proposition:
We put 
\begin{eqnarray}
\label{BM9}
\left(1/2\pi i\right)^n(d\partial_{J}\log\vert s\vert^2)^{n}:=\left(1 /2\pi i\right)^n d\left(\partial_{J}\log\vert s\vert ^2\wedge (d\partial_{J}\log\vert s\vert^2)^{n-1}\right)
\end{eqnarray}

We obtain the following 

\begin{theorem} 
\label{ThBM10}
We have 
\begin{eqnarray*}
\left(1/ 2\pi i\right)^n\left(d\partial_{J}\log\vert s\vert^2\right)^{n}=Div(s)-c_n(D)+d\gamma
\end{eqnarray*} where the current  $d\gamma$ satisfies the following estimate on every compact $K$ :
\begin{eqnarray}
\label{BM10}
\vert(d\gamma,\phi)\vert\leq C_K(\vert\vert \bar\partial_{J,\omega}s\vert\vert_{\infty}\vert\vert d\phi\vert\vert_{\infty}+\vert\vert \phi\vert\vert_{\infty}).
\end{eqnarray}
\end{theorem} 
We consider now some special cases of Theorem \ref{ThBM10}.

\subsection{Some special cases}
We begin with the following 
\begin{cor}
\label{corBM11} Let $s$ be an atomic section and  $\bar\partial_{J,\omega}s = 0 $ on $M$. then
\begin{eqnarray*}
\left(1/2\pi i\right)^n \left(d\partial_{J}\log\vert s\vert ^2\right)^n=Div(s)-c_n(D)+d\gamma
\end{eqnarray*} where $d\gamma$ is in  $L^{1}_{loc}$.
\end{cor}

\begin{proof} We repeat the previous argument and observe that  $\tau_2$ vanishes identically. Here we can assume that $s$ is just atomic.\end{proof}
\medskip

If furthermore a connection is flat, we obtain:
\begin{cor} Assume that  $D$ is flat and  $s$ is a holomorphic section. Then
\begin{eqnarray*}
 \left(1/ 2\pi i\right)^n \left(\bar\partial_{J} \partial_{J}\log \vert s\vert ^2\right)^n=Div(s).
 \end{eqnarray*}
\end{cor}

\begin{proof} Fix a local (not necessarily hermitian) frame $(f)$ parallel for the connection. Then $s = u f$ and 
 $\langle Ds,s\rangle$, $\langle Ds,Ds\rangle$ and $\langle D^2s,s\rangle$ can be written in this base. One easily obtain that 
 \begin{eqnarray*}
\left(1/2\pi i\right)^n \partial_{J}\log \vert s\vert^2\wedge(\bar\partial_{J} \partial_{J}\log \vert s\vert ^2)^{n-1} = \left(1/2\pi i\right)^n\beta\wedge (d\beta)^{n-1}
\end{eqnarray*} with  $\beta$ given by (\ref{BM1}). So we apply (\ref{BM2}). \end{proof}
\medskip

\begin{cor} 
\label{CompMan}
Suppose that $s$ is a regular (i.e. with the non-degenerate differential) section of an almost hermitian vector bundle. Furthermore assume that $\bar\partial_{J,\omega}s=0$ along $Z$. Then
\begin{eqnarray*}
\left(1/2\pi i\right)^n (\bar\partial_{J} \partial_{J}(\log \vert s\vert ^2)^n=Div(s)-c_n(D)+\theta
\end{eqnarray*}
where $\theta$ is in $L^{1}_{loc}(X)$.
\end{cor}
 This generalizes the result of  \cite{EL} which considers the case of the trivial bundle with flat connection.

Consider now the case of a hermitian vector bundle over a compact Kahler manifold. Of course,  we use the Chern connection. Let $s$ be an atomic section. Then Corollary \ref{corBM11} implies
\begin{eqnarray*}
\left(1/2\pi i\right)^n (\bar\partial_{J} \partial_{J}\log \vert s\vert ^2)^n=Div(s)-c_n+d\gamma
\end{eqnarray*}
and $d\gamma$ is in $L^1_{loc}$. On a compact Kahler manifold a d-exact form (current) of  type $(p>1,q>1)$ is  also  $\partial\bar\partial$-exact (see \cite{DE}). Since $d\gamma$ is a $(n,n)$ form (recall that  $Div(s)$ is a  $(n,n)$ form as well as the class $c_n$ ) it can be written in the form  $\partial\bar\partial G$. We obtain the following Green type formula: 
\begin{cor} We have
\begin{eqnarray*}
\left(1/2\pi i\right)^n (\bar\partial_{J} \partial_{J}\log \vert s\vert ^2)^n=Div(s)-c_n+\partial\bar\partial G
\end{eqnarray*} 
where $\partial\bar\partial \Omega$ is in  $L^1_{loc}$.
\end{cor}
This result implies that the divisor of $s$ represents the Chern class of the hermitian bundle in the Bott-Chern cohomology. The formula in the last corollary was well known in a more general setting : for complex manifold in general and for non atomic holomorphic sections (see \cite{ME1}, \cite{ME2}, \cite{AND}). But the proofs use the resolution of singularities theorem wich is not the case here : the theory of singular connections is much more elementary.    
%On en d\'eduit ais\'ement que la classe de chern de degr\'e maximal du fibr\'e est un repr\'esentant de la classe de $Div(s)$ dans le groupe de cohomologie de Bott-Chern ad-hoc. Soyons honn\^etes, ce type de formule est bien connue m\^eme dans le cas non Kalherien et pour des sections qui ne sont pas des intersections compl\'etes (par exemple, \cite{ME1} and \cite{ME2}). Mais la preuve utilisant les r\'esultats de \cite{HL1} l'est \`a notre avis et \'evite le th\'eor\`eme de r\'esolution des singularit\'es qui est n\'ecessaire m\^eme si $s$ est \`a z\'eros simples (voir \cite{ME1}). 

%Enfin, on peut remarquer que la preuve donne la m\^eme formule de Green pour une connexion $D$ qui n'est pas la connexion de Chern, \`a condition que $c_n(D)$ n'ait pas de composantes $(2n,0)$ et $(0,2n)$.

\subsection{Currents on Levi-flat CR manifolds}
Lelong-Poincar\'e type formulae obtained in the previous subsection could be used in study of currents supported on Levi-flat CR manifolds. These objects often arise in the theory of foliations. We consider the most natural case of Levi-flat hypersurfaces in an almost complex manifold of dimension 2 since foliations of such type exist (at least locally) for any almost complex structure. Our result and method of proof are similar to the "second support theorem" for currents due to Demailly \cite{D}.

Let $(M,J)$ be an almost complex manifold of (complex) dimension 2. A $C^1$-smooth closed real submanifold $\Gamma \subset M$ of codimension 1 is called a 
{\it Levi-flat hypersurface} if it is foliated near every point by a real 1-parameter family of $J$-complex discs. Such a foliation often is called {\it the Levi foliation} of $\Gamma$.

\begin{prop}
\label{SecStrTh}
Let $\Gamma$ be a Levi-flat hypersurface in an almost complex manifold $(M,J)$ of complex dimension 2 and $\gamma$ be a $C^1$ smooth real curve. Suppose that there exists a $C^1$ submersion $F:\Gamma \to \gamma$ such that every fiber $L_t = F^{-1}(t)$, $t \in \gamma$ is a connected $J$-complex curve in 
$\Gamma$ i.e. $L_t$ is a leaf of the Levi foliation. Suppose that $T \in {\mathcal D}_{1,1}'(M)$ is a closed  current of order $0$ with $Supp T \subset \Gamma$. Then there exists a unique complex measure $\nu$ on $\gamma$ such that 
\begin{eqnarray}
\label{lam1}
T = \int_{t \in \gamma} [L_t] d\nu(t)
\end{eqnarray}
The current $T$ is positive if and only if the measure $\nu$ is positive.
\end{prop}
Formula (\ref{lam1}) means that the current $T$ act on every test (1,1) form $\phi$ by
\begin{eqnarray*}
\langle T, \phi \rangle = \int_{t \in \gamma} \left( \int_{L_t} \phi \right) d\nu(t)
\end{eqnarray*}
The currents of this type ("foliated cycles") play an important role in the theory of dynamical systems and foliations, see for instance \cite{GD}. We note that the notion of a positive current in ${\mathcal D}_{p,p}'(M)$ on an almost complex manifold $M$ is defined similarly to the case of complex manifolds.

\begin{proof} First we show the uniqueness of $\nu$. Fix a compact subset $K$ in $\gamma$ and a $C^1$ retraction $\sigma: V \to \Gamma$ from a neighborhood $V$ of $\Gamma$ on $\Gamma$. Using the partition of unity, one can easily  construct a test (1,1) form $\psi$ with a support in $V$ 
such that $\langle [L_t], \psi \rangle = 1$ for all $t \in K$. Then for every continuous function $f$ with  support in $K$  we have 
\begin{eqnarray*}
\int_\gamma f(t) d\nu(t) = \langle T, (f \circ F \circ \sigma) \psi \rangle
\end{eqnarray*}
This implies the uniqueness and the positivity of $\nu$.

We prove the existence. Fix a point $q \in \Gamma$ and an open neighborhood $U$ of $q$ such that an open piece $\gamma' = F(\Gamma \cap U) \subset \gamma$ is covered by single coordinate chart. By the rank theorem $\Gamma \cap U$ is foliated by a real 2-parametric family of curves such that every curve is transverse to the leaves $L_t$ and is diffeomorphic by $F$ to $\gamma'$. We denote again by $t$ a local coordinate on every curve of this family defined as the pull-back by $F$ of a local coordinate on $\gamma'$. Let $\Gamma \cap U = \rho^{-1}(0)$ where $\rho$ is a $C^1$ real function with non-vanishing gradient in $U$. Consider the (1,0) form 
$\beta = \partial_J\rho$. The same argument as in Proposition \ref{FirstStrTh} shows that $\beta \wedge T = \overline\beta \wedge T = 0$. Therefore
$$T = T_{1,1}\beta \wedge \overline\beta$$
on $U$. Note that the forms $\beta$, $dt$ and $d\rho$ vanish on tangent spaces of the leaves $L_t$. Therefore the forms $\beta \wedge \overline\beta$ and $dt \wedge d\rho$ are proportional and the exists a complex measure $\lambda$ with support in $\Gamma \cap U$ such that 
\begin{eqnarray*}
T = \lambda dt \wedge d\rho
\end{eqnarray*}
on $\Gamma \cap U$. Every leaf $L_t$ can be locally parametrized via a $J$-complex disc by a coordinate $z = x + i y \in \C$. 
The current $T$ is closed, therefore $\partial\lambda/\partial x = \partial\lambda/\partial y = 0$. This means that the measure $\lambda$ depends only on $t$. Since it is supported on $\rho = 0$, one can write in the form  that is $\lambda = d\nu_U(t) \otimes \delta_0(\rho)$ where $d\mu_U$ is a measure defined on $\gamma' = F(\Gamma \cap U)$ and $\delta_0$ is the Dirac mass at the origin. Hence 
\begin{eqnarray*}
\langle T, \phi \rangle = \int_{t \in \gamma} [L_t] d\nu_U(t)
\end{eqnarray*}
on $\Gamma \cap U$. The established above uniqueness shows that for two open neigborhoods $U_1$ and $U_2$ the measures $d\nu_{U_1}$ and $d\nu_{U_2}$ coincide on $F(\Gamma \cap U_1  \cap U_2$. Since the fibers $L_t$ are connected, the existence of of the unique measure $\nu$ on $\gamma$ follows.
\end{proof}

This argument can easily carried out to the general case. This leads to the following 
\begin{prop}
\label{SecSupTh2}
Let $(M,J)$ be an almost complex manifold and $N$ be a closed CR submanifold of class $C^1$ in $M$ of CR dimension $p$. Suppose that there exists a submersion $F:N \to \gamma$ of class $C^1$ onto a real manifold $\gamma$ such that every fiber $L_t = F^{-1}(t)$, $t \in \gamma$ is  a connected $J$-complex subbmanifold if $N$ of complex dimension $p$ (i.e. $N$ is Levi-flat). Let also $T \in {\mathcal D}_{p,p}'(M)$ be a closed current of order $0$ with $supp T \subset N$. 
Then there exists a unique complex measure $\nu$ on $\gamma$ such that $T$ is represented by (\ref{lam1}). The current $T$ is positive if and only if the measure $\nu$ is positive.
\end{prop}
Under the assumptions of Proposition \ref{SecSupTh2}, Corollary \ref{CompMan} can be used for represantation of currents $[L_t]$ of integration over the Levi leaves.

\section{Convergence of divisors}

 In this section we consider the problem of convergence of currents of integration over zero sets of almost holomorphic sections. 
  \subsection{Line bundles: Donaldson's convergence theorem}
 Let  $(M,\Omega,J)$ be a compact  manifold with a symplectic form $\Omega$ and a compatible almost complex structure $J$.  As we pointed out previously, $\Omega$ and $J$ define canonically a riemannian metric $g$. Let  $L\longrightarrow M$ be a complex line bundle over  $M$.  Consider  $L^{\otimes k}$ the k-th  tensor power of $L$. We suppose that $[\Omega/2\pi] \in H^2(M,\R)$ lies in the integral lattice $H^2(M,\Z)/torsion$. Then there exists a connection $D$ on $L$ with the curvature equal to   $-i\Omega$. Then  $L^{\otimes k}$ is naturally equipped with the product connection $D^k$ curvature  equal to  $-ik\Omega$. Together with $J$ this curvature form defines on  $M$ the dilated metric  $g_k=kg$. 
 In what follows $C$, $C_j$ denote positive contants independent of $k$; they can change from line to line.
 
 The following result is due to  Donaldson \cite{D}:
\begin{prop} 
\label{ASH} There exists a sequence $(s_k)$ of sections of  $L^{\otimes k}$ such that on every unit ball defined by the metric $g_k$ the following holds:
\begin{itemize}
\item[(i)] $  \vert s_k\vert \leq C_1$. 
\item[(ii)] $\vert D^ks_k\vert\leq C_1, \,\,\, \vert D^kD^ks_k\vert\leq C_1$.
\item[iii)] $\vert\bar\partial_L (s_k)(z)\vert\leq C_1k^{-{1\over 2}},\,\,\, \vert D^k\bar\partial s_k\vert\leq C_1k^{-{1\over 2}}$. 
\item[(iv)] $\vert\partial_L (s_k)(z)\vert\geq C_2\  \hbox{if}\  \vert s_k(z)\vert\leq C_2.$
\end{itemize}
Here the norms are taken with respect to $g_k$.
\end{prop}
  The conditions  $(iii)$ et $(iv)$ imply that the zero set $Z_k = (s^k)^{-1}(0)$ of every section  $s_k$ is a smooth real submanifold of  
  codimension 2.

Consider the connection $D = s_k^{-1}D^k s_k$ with the potential $\tau_k$ is defined by (\ref{gauge_pot}) using the section $s_k$ and the connection $D^k$. Now we can write the fundamental equation (\ref{LP1}) for every section $s_k$ of $L^{\otimes k}$:

 \begin{eqnarray}
 \label{DA0}
\frac{1}{2\pi i} d\tau_k = [Z_k] - (k/2\pi)\Omega
 \end{eqnarray}

As Donaldson proved,  $\frac{1}{k} d\tau_k \longrightarrow 0$ (in the sense of currents) as $k \longrightarrow \infty$ (we present his argument in a more general form below). Hence we obtain the following

 \begin{cor}
 \label{Conv1} ${1\over k}[Z_k]\longrightarrow {\Omega/2\pi}$ in the sense of currents.\end{cor}
If the almost complex structure $J$ is integrable, Donaldson's constructure allows to construct holomorphic sections. In this case similar results on approximation of currents by holomorphic divisors were obtained by Duval-Sibony and Guedj \cite{DS,G}.

\subsection{General case}

In this section we consider asymptotically holomorphic  sections $s_k$ constructed by Auroux (this generalizes Donaldson's construction for line bundles to the case of vector bundles of arbitrary rang). We  prove that their potentials  $\sigma_k$  in the sense of   Harvey-Lawson  satisfy the following property : ${1\over k^n}d\sigma_k \longrightarrow 0$ when $k \longrightarrow \infty$.

We recall the generalization of Donaldson's construction due to  Auroux (\cite{AU}). We suppose that  $(M,\Omega,J)$ is a compact symplectic manifold of real dimension $2m$  and  satisfying the assumptions of Donaldson's theorem. Let   $E \longrightarrow M$ be a  hermitian vector bundle  of rang $n$, $n \le m$. Consider a complex line bundle $L \longrightarrow M$ with a hermitian connection $D$ such that its curvature is equal to $-i\Omega$. We again consider the k-th tensor power   $L^{\otimes k}$ with the connection $D^k$ and the curvature $-ik\Omega$. Note that the choice of connection on $E$ is not essential and only changes the constants in the estimates below. These hermitian connections define on the product  $E\otimes L^k$ a hermitian connection  which (with a slight abuse of notations) we again denote by $D^k$. 
  %\footnote{Nous devrions \'ecrire $D^k$ car la connexion sur $E\otimes L^k$ d\'epend de $k$...mais nous choisissons de continuer \`a \'ecrire $D$  quand nous travaillons pour $k$ fix\'e. Quand nous ferons tendre $k$ vers l'infini, \`a la fin de la section, nous en tiendrons compte.}  
  
  A sequence of sections  $(s_k)_k$ of $E\otimes L^k$ is called {\it asymptotically holomorphic} if it satisfies the following properties (i) - (iii) of Proposition  \ref{ASH}; the property (iv) needs to be adapted to the case of vector functions:
\begin{itemize} 
\item[(iv)] there exists  $\eta > 0$ (independent of $k$) such that  the covariant derivative $D^ks_k(x): T_xM \longrightarrow(E\otimes L^k)_x$ is surjective and admits a right inverse whose norm  is smaller than $\eta^{-1}$, for all $x$ satisfying $\vert s_k(x)\vert\leq \eta$ and all $k$ large enough.
\end{itemize}
The property (iv) is called the $\eta$-transversality. Note that the norms are considered with respect to the fixed metric $g_k$.

The main result of this section is

\begin{theorem} 
\label{MainConvTheo}Let  $(s_k)_k$ be a sequence of asymptotically holomorphic sections satisfying (i)-(iv). Then $(1/ k^n)Div(s_k) \longrightarrow \Omega^n /(2\pi)^n$ in the sense of currents when $k \longrightarrow \infty$.
\end{theorem}
This is shown in \cite{AU} using topological arguments that the zero sets of sections $s_k$ have some density propeties when $k$ goes to $\infty$; to our best knowledge the approximation Theorem \ref{MainConvTheo} did not appear in literature previously. 

First we establish a general estimate for the Bochner-Martinelli potential.

 Let $E \longrightarrow M$ be a vector bundle of  rang $n$ over a smooth manifold $M$. We suppose that $E$ is equipped with  a hermitian connection   $D$ with the curvature  $\Omega$. Let $s$ be a smooth (global) section of $E$. We have the following:

\begin{lemma}
\label{DA1} The potential   $\sigma$ of $s$ defined by  (\ref{BM4}) satisfy the following estimate:
\begin{eqnarray}
\label{DA2}
\vert \sigma\vert\leq C \left({\vert Ds\vert \over \vert s\vert}\sum_{j=0}^{n-1}\left(\vert \Omega\vert+ {\vert Ds\vert^2\over \vert s\vert^2}\right)^{j}\vert \Omega\vert^{n-1-j}\right).
\end{eqnarray}
\end{lemma}
Note that all norms here are defined by the hermitian structure of the bundle $E \to M$

\begin{proof} Since the potential $\sigma$ is independent of a choice of a  base, we choose an orthonormal frame  $f$ in a local trivialization. Then $h=Id$, $u^*=\bar u^t$, $Du=du+u\omega$ et $Du^{*}=du^{*}-\omega u^{*}$. Therefore in order to show that   $\sigma$ satisfies (\ref{DA2})  it suffices to establish the estimates:
\begin{eqnarray}
\label{DA3} 
 \vert f^*u^*Duf\vert\leq \vert Ds\vert\vert s\vert
 \end{eqnarray}
 and
 \begin{eqnarray}
 \label{DA4}
 \vert f^*Du^*Duf\vert\leq \vert Ds\vert^2.
 \end{eqnarray}
 The equalities  $\vert f^*u^*\vert =\vert \sum f^*\bar u^t\vert=\vert s\vert$ and  $\vert Duf\vert=\vert Ds\vert$ imply (\ref{DA3}). Since the connection $D$ is hermitian and the frame $f$ is orthonormal, the connection form $\omega$ is anti-hermitian. Hence we have  $\vert f^*Du^*\vert=\vert f^*(du^{*}-Au^{*})\vert=\vert \overline{(Ds)^t}\vert$ and (\ref{DA4}) follows. This proves lemma.\end{proof}
\bigskip

It follows by Lemma \ref{DA1} that  the potential $\sigma_k$ of $s_k$ satisfies the uniform (in $k$) estimate:
\begin{eqnarray}
\label{DA5}
\vert\sigma_k\vert\leq C {k^{n-{1\over 2}}\over \vert s_k\vert ^{2n-1}}.
\end{eqnarray} This estimate will be used in the proof of the main technical result of this section:

\begin{prop}
\label{DA6} For every test form  $\phi$ on $M$, the potential  $\sigma_k$ satisfies the estimate
\begin{eqnarray}
\label{DA7}
{1\over k^n}\vert \langle d\sigma_k,\phi\rangle  \vert\leq C k^{-{1\over 2}}\vert\vert d\phi\vert\vert_{\infty}
\end{eqnarray}
uniformly in $k$.
\end{prop}

\begin{proof} We closely follow Donaldson's argument for the case of line bundles. From (\ref{DA5}) we obtain 
\begin{eqnarray}
\label{DA8}
{1\over k^n}\vert \langle d\sigma_k,\phi \rangle \vert\leq C k^{-{1\over 2}}\vert\vert d\phi\vert\vert_{\infty}\int_{M}{1\over \vert s_k\vert^{2n-1}}d\lambda
\end{eqnarray}
where $\lambda$ is the  volume form on $M$. 

It suffices to show that the integral in the right hand is  upper bounded uniformly in $k$.  Note that the norm of the section $s_k$ in (\ref{DA5}) and in the integral in (\ref{DA8}) is defined by the metric $g_k = kf$.

We cover $M$ by the unit balls $B_k$ (we do not specify their centers in these notations) with respect to the dilated metric $g_k$ on $M$; such a covering contains $O(k^n)$ balls. Denote by  $d\lambda_k = k^n d\lambda$ 
the volume form of $g_k$.   It suffices to show that the following upper bound uniform in $k$:
\begin{eqnarray}
\label{DA9}
I = \int_{B_k}{1\over \vert s_k\vert^{2n-1}}d\lambda_k \le C
\end{eqnarray}

For $k$ large enough we can choose on the ball $B_k$ a trivialization of the bundle  $E\otimes L^k$ such that the connection matrix of  $L^k$ is given by  $\omega ={1\over 4}(\sum z_id\bar z_i-\bar z_idz_i)$ (in fact we first chosse a trivialization with the connection form  $k\omega$ and perform a dilation of coordinates with the coefficient $k^{-{1\over 2}})$. In such a trivialization a connection fixed on $E$ is irrelevant since it only changes constants in estimates (recall that  $M$ is compact) and we can identify $B_k$ with the Euclidean unit ball $B$. Then every section $s_k$ can be viewed as a vector function  $\sum_{i=1}^{n}u^{k}_ie_i$ where $e_i$ is a usual base of $\C^n$ and   $u^{k}_i: B \longrightarrow \C$ are almost holomorphic functions. Now the desired estimate follows by Donaldson's estimate of the integral (\ref{DA9}) (obtained for $n=1$ and scalar almost holomorphic functions satisfying the $\eta$-transversality) applied to every component $u^k_i$. The vector functions $u^k$ satisfy the following properties:

\begin{itemize}
\item[(i')] $\vert u^k\vert=O(1),\ \vert \bar\partial u^{k}_i\vert =O(k^{-{1\over 2}})$
\item[(ii')] $\vert d u^{k}_i \vert=O(1)$
\item[(iii')] $\vert d\bar\partial u^{k}_i \vert=O(1)$
\end{itemize}
which follow respectively  from the properties (i), (ii), (iii) of asymptotically holomorphic functions. By the Jacobian determinant $J u$ of a vector function $u$ we mean 
$$J u:=  (\det (\partial u)(\overline{\partial u})^t)^{1/2}$$
where $(\partial u) = (\partial u_j/ \partial z_s)$ is the (complex) Jacobian matrix of $u$ and $(\bullet)^t$ denotes the transposition. 

The $\eta$-transversality (iv) implies that 
\begin{itemize}
\item[(iv')] $(Ju^{k})^{-1}\leq \eta^{-1}$ if  $\vert u^{k}\vert\leq  \eta$ 
\end{itemize}

Denote by  $m$ the dimension of $M$. 

\begin{lemma}
Suppose that a vector function $u^{k}=(u^{k}_1,\cdots ,u^{k}_n)$ smooth in a neighborhood of the closed unit ball $B$ satisfies (i')-(iv'). Then the integral  
$$I = \int_{B}{1\over \vert u^{k}\vert^{2n-1}}d\lambda \le C$$
is bounded uniformly in $k$.
\end{lemma}

\begin{proof} Fix  $\eta > 0$ provided by (iv'). We have $I = I_1 + I_2$ with 
$$I_1 =  \int_{B\cap \{\vert u^{k}\vert\geq \eta\}}{1\over \vert u^{k}\vert^{2n-1}}d\lambda\le  \eta^{1-2n}Vol(B)$$ and 
 $$I_2 = \int_{B\cap \{\vert u^{k}\vert\leq \eta\}}{1\over \vert u^{k}\vert^{2n-1}}d\lambda.$$ 
  In order to find an upper bound for the second integral, we follow Donaldsson's argument and apply the co-area formula 
  $$I_2 = I_3$$ with 
  \begin{eqnarray*}
I_3= \int_{\vert \sigma\vert \leq \eta}\bigg(\int_{B\cap \{u^{k}=\sigma\}}{1\over \vert u^k\vert^{2n-1}}.(Ju^{k})^{-1}d\lambda_{\sigma}\bigg)d\sigma 
\end{eqnarray*}
Here $\sigma=(\sigma_1,\cdots,\sigma_n)$ and $d\lambda_\sigma$ are volume measures for the level set $\{u^{k}=\sigma\}$; note that by (iv') this level set is a smooth variety of codimension 2n  when $\vert\sigma\vert\leq \eta$.  
In view of the bound (iv') we have 
\begin{eqnarray*}
I_3 = \int_{\vert \sigma\vert \leq \eta}\bigg(\int_{B\cap \{u^{k}=\sigma\}}{1\over \vert \sigma\vert^{2n-1}}.(Ju^{k})^{-1}d\lambda_{\sigma}\bigg)d\sigma
 \le C\int_{\vert \sigma\vert \leq \eta}{Vol(B\cap \{ u^k=\sigma\})\over \vert \sigma\vert ^{2n-1}}.
 \end{eqnarray*} 
 
 In order to conclude the proof, it suffices to establish an upper bound on the  volumes of manifolds $B\cap \{ u^k=\sigma\}$ 
 uniform  $k$ and $\vert \sigma\vert \leq \eta$. Arguing by contradiction, assume that (passing to a subsequence) that $(u^k)$ satisfies (i')-(iv') and 
 a sequence   $(\sigma^k)$ with $\vert \sigma_k\vert\leq \eta$ are such that  $B\cap \{ u^k=\sigma_k\}\rightarrow\infty$ when $k \to \infty$. 
 Using the standard elliptic estimates for the $\overline\partial$-operator (for example, the  Bochner-Martinelli integral formula for every component of $u^k$ in $B$) $u^{k}_i$, we obtain from the properties (i') and (iii') that the sequence $(u^{k})$ is uniformly bounded in the  $C^2$ norm on  $B$. By Arzela-Ascoli, the (sub) sequence $(u^{k})_k$ converges to $u$ in the  $C^1$-norm on  $B$; one can assume also that $\sigma^k$ converges to $\sigma$ with $\vert \sigma\vert\leq \eta$. Since  $u$ satisfies (ii'), (iv'), the volume  $Vol(B\cap u=\sigma)$ is finite. But $B\cap \{ u^k=\sigma_k\}\rightarrow Vol(B\cap u=\sigma)$ which is a contradiction. This proves lemma \end{proof} 
 Now proposition follows.\end{proof}

Using the fundamental equation (\ref{BM5}) we deduce now that  
\begin{eqnarray*}{1\over k^n}(Div(s_k)-c_n(D^k))\rightarrow 0 \,\, \mbox{ when} \,\, k\rightarrow\infty.
\end{eqnarray*}

 On the other hand it is easy to see using the Chern classes properties of $D^ k$ that (see for example \cite {AU} pp. 991)
 \begin{eqnarray}
 c_r(D)=C^{r}_{n} {\Omega^r\over (2\pi)^r} + O(k^{r-1}),\ \   1\leq r\leq n,
 \end{eqnarray}
 and so 
 \begin{eqnarray}
 \label{classconv}
 {1\over k^n}c_n(D^k)\rightarrow {\Omega^n\over (2\pi)^n}.
 \end{eqnarray}
 This proves the theorem.

%De ceci, nous pouvons d\'eduire un r\'esultat int\'eressant \`a notre avis. Si nous consid\'erons l'\'egalit\'e des courants suivants de la section 6:
%$$d\big({[Ds_k,s_k]\over \vert s_k\vert^2}\wedge \tau_{1,k}^{n-1}\big)-1_{M/Z_k}d\big({[Ds_k,s_k]\over \vert s_k\vert^2}\wedge \tau_{1,k}^{n-1}\big)=(2\pi i)^n Div(s_k),$$ o\`u $\tau_{1,k}={1\over \vert s_k\vert ^2}([D^2s_k,s_k]-[Ds_k,Ds_k])$ et $Z_k$ est l'ensemble des z\'eros de $s_k$,
%il est alors facile de voir que 
%$$\bigg\vert{[Ds_k,s_k]\over \vert s_k\vert^2}\wedge \tau_{1,k}^{n-1}\bigg\vert\leq {k^{n-{1\over 2}}\over \vert s_k\vert^{2n-1}}.$$ Par cons\'equent, ${1\over k^n}d\big({[Ds_k,s_k]\over \vert s_k\vert^2}\wedge \tau_{1,k}^{n-1}\big)$ tend vers z\'ero au sens des courants quand $k$ tend vers l'infini. La proposition 7.4 donne alors le corollaire :
%\begin{cor}La forme localement int\'egrable ${1\over k^n}.1_{M/Z_k}d\big({[Ds_k,s_k]\over \vert s_k\vert^2}\wedge \tau_{1,k}^{n-1}\big)$ tend vers $(iw)^n$ quand $k$ tend vers l'infini au sens des courants.
%\end{cor} 

\end{document}